\documentclass{birkjour}
 \newtheorem{thm}{Theorem}[section]
 \newtheorem{cor}[thm]{Corollary}
 \newtheorem{lem}[thm]{Lemma}
 \newtheorem{prop}[thm]{Proposition}
 \theoremstyle{definition}
 \newtheorem{defn}[thm]{Definition}
 \newtheorem{ex}[thm]{Example}
 \theoremstyle{remark}
 \newtheorem{rem}[thm]{Remark}
 \numberwithin{equation}{section}
\usepackage{color}

\begin{document}

%
%
%
%
%
%
%
%
%

\title[Compare and contrast between duals of fusion and discrete frames]
 {Compare and contrast between duals of fusion and discrete frames}


\author[E. Osgooei ]{Elnaz Osgooei}
\address{Department of sciences,
Urmia University of Technology, Urmia, Iran.} \email{
e.osgooei@uut.ac.ir; osgooei@yahoo.com}
\author[A. Arefijamaal]{Ali Akbar Arefijamaal}
\address{Department of Mathematics and Computer Sciences, Hakim Sabzevari University, Sabzevar, Iran.}
\email{arefijamaal@hsu.ac.ir; arefijamaal@gmail.com}


\vspace{2cm}
\begin{abstract}
Fusion frames are valuable generalizations of discrete frames. Most
concepts of fusion frames are shared by discrete frames. However,
the dual setting is so complicated. In particular, unlike discrete
frames, two fusion frames are not dual of each other in general. In
this paper, we investigate the structure of the duals of fusion
frames and discuss the relation between the duals of fusion frames
with their associated discrete frames.
\end{abstract}

\maketitle
\section{Introduction and preliminaries}
Fusion frames, as a generalization of frames, are valuable tools to
subdividing a frame system into smaller subsystems and combine
locally data vectors. The theory of fusion frames was systematically
introduced in \cite{Cas04, Cas08}. Since then, many useful results
about the theory and application of fusion frames have been obtained
rapidly \cite{CA, Cas09, Ka94, Mas10}.

In the context of signal transmission, fusion frames and their
alternative duals have important roles in reconstructing signals
in terms of the frame elements. The duals of fusion frames for
experimental data transmission are investigated in \cite{Leng1}.
But the problem that occurs is that the duality properties of
fusion frames are not like discrete frames, such as, the duality
property of fusion frames is not alternative and fusion Riesz
bases have more than one dual. This paper deals with investigating
such problems, which help us to obtain alternative dual fusion
frames.

Let $\mathcal{H}$ be a separable Hilbert space. A \textit{frame}
for $\mathcal{H}$ is a sequence
$\{f_{i}\}_{i=1}^{\infty}\subseteq\mathcal{H}$ such that there are
constants $0<A\leq B<\infty$ satisfying
\begin{eqnarray}\label{Def frame}
A\|f\|^{2}\leq \sum_{i=1}^{\infty}|\langle f,f_{i}\rangle|^{2}\leq
B\|f\|^{2},\qquad f\in \mathcal{H}.
\end{eqnarray}
The constants $A$ and $B$ are called \textit{frame bounds}. If
$A=B$, we call $\{f_{i}\}_{i=1}^{\infty}$ a \textit{tight frame}.
If the right-hand side of (\ref{Def frame}) holds, we say that
$\{f_{i}\}_{i=1}^{\infty}$ is a \textit{Bessel sequence}. Given a
frame $\{f_{i}\}_{i=1}^{\infty}$, the \textit{frame operator} is
defined by
\begin{eqnarray*}
Sf=\sum_{i=1}^{\infty}\langle f,f_{i}\rangle f_i.
\end{eqnarray*}
A direct calculation yields
\begin{eqnarray*}
\langle Sf,f \rangle = \sum_{i\in I}|\langle f,f_{i}\rangle|^{2}.
\end{eqnarray*}
Hence, the series defining $Sf$ converges unconditionally for all
$f\in\mathcal{H}$ and $S$ is a bounded, invertible, and
self-adjoint operator. Hence, we obtain
\begin{eqnarray}\label{frame
decomposition} f=S^{-1}Sf=\sum_{i=1}^{\infty}\langle
f,S^{-1}f_{i}\rangle f_i ,\qquad f\in \mathcal{H}.
\end{eqnarray}
The possibility of representing every $f\in\mathcal{H}$ in this
way is the main feature of a frame. A sequence
$\{f_{i}\}_{i=1}^{\infty}$ is Bessel sequence if and only if the
operator
$T:\ell^{2}\rightarrow\mathcal{H};\{c_{i}\}\mapsto\sum_{i=1}^{\infty}
c_i f_i$, which is called the \textit{synthesis operator}, is
well-defined and bounded. When $\{f_{i}\}_{i=1}^{\infty}$ is a
frame, the synthesis operator T is well-defined, bounded and onto.
A sequence $\{g_i\}_{i=1}^{\infty}\subseteq\mathcal{H}$ is called
a \textit{dual} for Bessel sequence $\{f_{i}\}_{i=1}^{\infty}$ if
\begin{eqnarray}\label{Def Dual}
f=\sum_{i=1}^{\infty}\langle f,g_{i}\rangle f_i ,\qquad f\in
\mathcal{H}.
\end{eqnarray}
Every frame at least has a dual. In fact, if
$\{f_{i}\}_{i=1}^{\infty}$ is a frame, then (\ref{frame
decomposition}) implies that $\{S^{-1}f_{i}\}_{i=1}^{\infty}$,
which is a frame with bounds $B^{-1}$ and $A^{-1}$, is a dual for
$\{f_{i}\}_{i=1}^{\infty}$; it is called the \textit{canonical
dual}. To see a general text in frame theory see \cite{Chr08}.

Let $\{f_i\}_{i=1}^{\infty}$ and $\{g_i\}_{i=1}^{\infty}$ be
Bessel sequences with synthesis operators $T$ and $U$,
respectively. Then from (\ref{Def Dual}) follows immediately that
$\{f_i\}_{i=1}^{\infty}$ and $\{g_i\}_{i=1}^{\infty}$ are dual of
each other if and only if $UT^{*}=I_{\mathcal{H}}$; in particular,
they are frames. For more studies in the duality properties of
frames we refer to \cite{Ar12, Ar13, Ole01, Deh, Gav07, Eli1}.

The following proposition describes a characterization of
alternate dual frames.
\begin{prop} \cite{Ar13, Chr08} \label{ar}
\begin{enumerate}
\item The dual frames of $\{f_i\}_{i=1}^{\infty}$ are precisely as
$\{\Phi\delta_i\}_{i=1}^{\infty}$, where
$\Phi:\ell^2\rightarrow\mathcal{H}$ is a bounded left inverse of
$T^{*}$ and $\{\delta_{i}\}_{i=1}^{\infty}$ is the canonical
orthonormal basis of $\ell^{2}.$ \item  There is a one to one
correspondence between dual frames of $\{f_i\}_{i=1}^{\infty}$ and
operators $\Psi\in B(\mathcal{H},\ell^2)$ such that $T\Psi=0$.
\end{enumerate}
\end{prop}
We now review preliminary results about fusion frames. Throughout
this paper, $I$ denotes a countable index set and $\pi_V$ is the
orthogonal projection onto a closed subspace $V$ of $\mathcal{H}$.
\begin{defn} Let $\{W_i\}_{i\in I}$ be a family of closed subspaces of $\mathcal{H}$ and $\{\omega_i\}_{i\in I}$ be a family of
weights, i.e. $\omega_i>0$, $i\in I$. Then
$\{(W_i,\omega_i)\}_{i\in I}$ is a \textit{fusion frame} for
$\mathcal{H}$ if there exist constants $0<A\leq B<\infty$ such
that
\begin{eqnarray}\label{Def fusion}
A\|f\|^{2}\leq \sum_{i\in I}\omega_i^2\|\pi_{W_i}f\|^2\leq
B\|f\|^{2},\qquad f\in \mathcal{H}.
\end{eqnarray}
\end{defn}
The constants $A$ and $B$ are called the \textit{fusion frame
bounds}. If we only have the upper bound in (\ref{Def fusion}) we
call $\{(W_i,\omega_i)\}_{i\in I}$ a \textit{Bessel fusion
sequence}. A fusion frame is called \textit{tight}, if $A$ and $B$
can be chosen to be equal, and \textit{Parseval} if $A = B = 1$.
If $\omega_i =\omega$ for all $i\in I$, the collection
$\{(W_i,\omega_i)\}_{i\in I}$ is called \textit{$\omega$-uniform}.
A fusion frame $\{(W_i,\omega_i)\}_{i\in I}$ is said to be an
\textit{orthonormal fusion basis} if $\mathcal{H} =\bigoplus_{i\in
I} W_i$, and it is a \textit{Riesz decomposition} of $\mathcal{H}$
if for every $f\in\mathcal{H}$ there is a unique choice of $f_i\in
W_i$ so that $f =\sum_{i\in I}f_i$.

Recall that for each sequence $\{W_i\}_{i\in I}$ of closed
subspaces in $\mathcal{H}$, the space
\begin{eqnarray*}
(\sum_{i\in I}\oplus W_i)_{\ell^2}=\{\{f_i\}_{i\in I}:f_i\in W_i,
\sum_{i\in I}\|f_i\|^2<\infty\},
\end{eqnarray*}
with the inner product
\begin{eqnarray*}
\langle \{f_i\}_{i\in I},\{g_i\}_{i\in I} \rangle=\sum_{i\in
I}\langle f_i,g_i \rangle,
\end{eqnarray*}
is a Hilbert space.

For a Bessel fusion sequence $\{(W_i,\omega_i)\}_{i\in I}$ for
$\mathcal{H}$, the \textit{synthesis operator} $T_{W}:(\sum_{i\in
I}\oplus W_i)_{\ell^2} \rightarrow\mathcal{H}$ is defined by
\begin{equation*}
T_{W}(\{f_i\}_{i\in I})=\sum_{i\in I}\omega_if_i,\ \
\{f_{i}\}\in(\sum_{i\in I}\oplus W_i)_{\ell^2}.
\end{equation*}
Its adjoint operator $T_{W}^{*}: \mathcal{H}\rightarrow(\sum_{i\in
I}\oplus W_i)_{\ell^2}$ which is called the \textit{analysis
operator} is given by
\begin{eqnarray*}
T_{W}^{*}(f)=\{\omega_{i}\pi_{W_{i}}(f)\},\ \ f\in
\mathcal{H}.\end{eqnarray*} Recall that $\{(W_i,\omega_i)\}_{i\in
I}$ is a fusion frame if and only if the bounded operator $T_{W}$
is onto \cite{Cas04} and its adjoint operator $T_{W}^{*}$ is
(possibly into) isomorphism. If $\{(W_i,\omega_i)\}_{i\in I}$ is a
fusion frame, the \textit{fusion frame operator}
$S_{W}:\mathcal{H}\rightarrow\mathcal{H}$ is defined by $S_{W
}f=T_{W}T^*_{W}f=\sum_{i\in I}\omega_i^{2}\pi_{W_i}f$ is a
bounded, invertible and positive operator and we have the
following reconstruction formula \cite{Cas04}
\begin{eqnarray*}
f=\sum_{i\in I}\omega_i^2S_W^{-1}\pi_{W_i}f,\qquad f\in
\mathcal{H}.
\end{eqnarray*}
The family $\{(S_{W}^{-1}W_i,\omega_i)\}_{i\in I}$, which is also
a fusion frame, is called the \textit{canonical dual} of
$\{(W_i,\omega_i)\}_{i\in I}$ and satisfies the following
reconstruction formula \cite{Gav07}
\begin{eqnarray*}
f=\sum_{i\in
I}\omega_i^2\pi_{S_{W}^{-1}W_i}S_{W}^{-1}\pi_{W_i}f,\qquad f\in
\mathcal{H}.
\end{eqnarray*}
\begin{defn} Let $\{(W_i,\omega_i)\}_{i\in I}$ be a fusion frame by the frame operator
$S_{W}$. A Bessel fusion sequence $\{(V_i,\nu_i)\}_{i\in I}$ is
called a \textit{dual} of $\{(W_i,\omega_i)\}_{i\in I}$ if
\begin{eqnarray}\label{Def:alt}
f=\sum_{i\in
I}\omega_{i}\nu_{i}\pi_{V_i}S_{W}^{-1}\pi_{W_i}f,\qquad f\in
\mathcal{H}.
\end{eqnarray}
\end{defn}
\begin{defn}
Let $\{W_i\}_{i\in I}$ be a family of closed subspaces of
$\mathcal{H}$ and $\{\omega_i\}_{i\in I}$ be a family of weights,
i.e. $\omega_i>0$, $i\in I$. We say that
$\{(W_{i},\omega_{i})\}_{i\in I}$ is a \textit{fusion Riesz basis}
for $\mathcal{H}$ if $\overline{span}_{i\in
I}\{W_{i}\}=\mathcal{H}$ and there exist constants $0< C\leq
D<\infty$ such that for each finite subset $J\subseteq I$
\begin{equation*}
C(\sum_{j\in J}\|f_{j}\|^{2})^{\frac{1}{2}}\leq\|\sum_{j\in
J}\omega_{j}f_{j}\|\leq D(\sum\|f_{j}\|^{2})^{\frac{1}{2}},\ \
f_{j}\in W_{j}.
\end{equation*}
\end{defn}
Some characterizations of fusion Riesz bases are given in the
following theorem.
\begin{thm}\label{rie}\cite{Cas04, naj1}
Let $\{(W_i,1)\}_{i\in I}$ be a fusion frame for $\mathcal{H}$ and
$\{e_{i, j}\}_{j\in J_{i}}$ be a basis for $W_{i}$ for each $i\in
I$. Then the following conditions are equivalent.
\\(1) $\{(W_i, 1)\}_{i\in I}$ is a Riesz decomposition of $\mathcal{H}$.
\\(2) The synthesis operator $T_{W}$ is one-to-one.
\\(3) The analysis operator $T_{W}^{*}$ is onto.
\\(4) $\{(W_{i}, 1)\}_{i\in I}$ is a fusion Riesz basis for $\mathcal{H}$.
\\(5) $\{e_{i, j}\}_{i\in I, j\in J_{i}}$ is a Riesz basis for $\mathcal{H}$.
\end{thm}
\begin{lem}\label{har}\cite{Gav07}
Let $T\in B(\mathcal{H})$ and $V\subseteq\mathcal{H}$ be a closed
subspace . Then we have
\begin{equation*}
\pi_{V}T^{*}=\pi_{V}T^{*}\pi_{\overline{TV}}.
\end{equation*}
\end{lem}
This paper is organized as follows: In Section 2, we compare the
duality properties of discrete and fusion frames and by presenting
examples of fusion frames we show that some well-known results on
discrete frames are not valid on fusion frames. Also we
investigate the cases that these properties can satisfy on fusion
frames. In Section 3, we investigate the relation between the
duals of fusion frames, local frames and the associated discrete
frames and we try to characterize the dual of fusion frames.
\section{Contrasting of dual of fusion frames}
For a fusion frame $\{(W_i,\omega_i)\}_{i\in I}$ and a Bessel
fusion sequence $\{(V_i,\nu_i)\}_{i\in I}$, we define
$\phi:(\sum_{i\in I}\oplus W_i)_{\ell^{2}}\rightarrow (\sum_{i\in
I}\oplus V_i)_{\ell^{2}}$ by
\begin{eqnarray}\label{Def:phi}
\phi(\{f_i\}_{i\in I})=\{\pi_{V_i}S_{W}^{-1}f_i\}_{i\in
I}.\end{eqnarray} It is easy to see that $\phi$ is a linear
operator and $\|\phi\|\leq \|S_W^{-1}\|$, its adjoint can be given
by $\phi^{*}(\{g_i\}_{i\in I})=\{\pi_{W_i}S_W^{-1}g_i\}_{i\in I}$,
for all $\{g_i\}_{i\in I}\in (\sum_{i\in I}\oplus
V_i)_{\ell^{2}}$. Now, the identity (\ref{Def:alt}) can be written
in an operator form as follows.
\begin{lem}\label{dual by preframes}
Let $\{(W_i,\omega_i)\}_{i\in I}$ be a fusion frame. A Bessel
fusion frame $\{(V_i,\nu_i)\}_{i\in I}$ is a dual of $\{(W_i,
\omega_{i})\}_{i\in I}$ if and only if
\begin{eqnarray}\label{u.phi.T*}
T_{V}\phi T^{*}_{W}=I_{\mathcal{H}},
\end{eqnarray}
where $T_{W}$ and $T_{V}$ are the synthesis operators of
$\{W_i\}_{i\in I}$ and $\{V_i\}_{i\in I}$, respectively.
\end{lem}
By Lemma \ref{dual by preframes}, we deduce that, unlike discrete
frames, two fusion frames are not dual of each other in general.
Here we present an example which confirms this statement.
\begin{ex}\label{S}
Let $I=\{1, 2, ..., 6\}$. Consider
\begin{equation*}
W_{1}=span\{(1, 0, 0)\},\ \ W_{2}=span\{(0, 1, 0)\},\ \
W_{3}=span\{(0, 0, 1)\},
\end{equation*}
\begin{equation*}
W_{4}=span\{(0, 1, 0)\},\ \ W_{5}=span\{(1, 0, 0)\},\ \
W_{6}=span\{(0, 0, 1)\},
\end{equation*}
and $\omega_{i}=1$, for $i\in I$. Also take
\begin{equation*}
V_{1}=span\{(1, 0, 0)\},\ \ V_{2}=span\{(0, 1, 0)\},\ \
V_{3}=span\{(0, 0, 1)\},
\end{equation*}
\begin{equation*}
V_{4}=span\{(0, 0, 1)\},\ \ V_{5}=span\{(0, 1, 0)\},\ \
V_{6}=span\{(1, 0, 0)\},
\end{equation*}
and $\nu_{i}=2$, for $i\in I$. Then $\{(W_{i},\omega_{i})\}_{i\in
I}$ and $\{(V_{i},\nu_{i})\}_{i\in I}$ are fusion frames for
$\mathbb{R}^{3}$ with frame operators $S_{W}=2I_{\mathbb{R}^{3}}$
and $S_{V}=8I_{\mathbb{R}^{3}}$, respectively. The following
calculation shows that $\{(V_{i}, 2)\}_{i\in I}$ is an alternative
dual of $\{(W_{i}, 1)\}_{i\in I}$.
\begin{eqnarray*}
\sum_{i\in I}\nu_{i}\omega_{i}\pi_{V_{i}}S_{W}^{-1}\pi_{W_{i}}(a,
b, c)&=&2[\frac{1}{2}(a, 0, 0)+\frac{1}{2}(0, b, 0)+\frac{1}{2}(0,
0, c)]\\&=&(a, b, c),\ \ \ (a, b, c)\in\mathbb{R}^{3}.
\end{eqnarray*}
But $\{(W_{i}, 1)\}_{i\in I}$ is not an alternative dual of
$\{(V_{i}, 2)\}_{i\in I}$. In fact
\begin{eqnarray*}
\sum_{i\in I}\nu_{i}\omega_{i}\pi_{W_{i}}S_{V}^{-1}\pi_{V_{i}}(a,
b, c)=2[\frac{1}{8}(a, 0, 0)+\frac{1}{8}(0, b, 0)+\frac{1}{8}(0,
0, c)]\neq (a, b, c).
\end{eqnarray*}
\end{ex}
Now, it is natural to ask when two Bessel fusion frames are dual
of each other. To answer this question assume that $\{(W_i,
\omega_{i})\}_{i\in I}$ is also a dual fusion frame for $\{(V_i,
\nu_{i})\}_{i\in I}$ or equivalently (by Lemma \ref{dual by
preframes})
\begin{eqnarray}\label{khan}
T_{W}\psi T_{V}^{*}=I_{\mathcal{H}},
\end{eqnarray}
where $\psi:(\sum_{i\in I}\oplus V_i)_{\ell^{2}}\rightarrow
(\sum_{i\in I}\oplus W_i)_{\ell^{2}}$ is given by
\begin{eqnarray*}
\psi(\{g_i\}_{i\in I})=\{\pi_{W_i}S_V^{-1}g_i\}_{i\in I}.
\end{eqnarray*}
\begin{prop} Let $\{(W_i, 1)\}_{i\in I}$ be a fusion frame with a dual $\{(V_i, 1)\}_{i\in I}$.
Then the fusion frame $\{W_i\}_{i\in I}$ is also a dual of
$\{V_i\}_{i\in I}$ if
\begin{eqnarray}\label{dual each other}
\phi^{*}=\psi.
\end{eqnarray}
Moreover the converse is hold if $\{W_{i}\}_{i\in I}$ and
$\{V_{i}\}_{i\in I}$ are fusion Riesz bases.
\end{prop}
\begin{proof}
Let $\{V_i\}_{i\in I}$ be a dual of $\{W_i\}_{i\in I}$, then by
using (\ref{u.phi.T*}) and (\ref{dual each other}) we obtain
\begin{eqnarray*}
\left\langle \sum_{i\in I}\pi_{W_i}S_V^{-1}\pi_{V_i}f,
f\right\rangle &=&\langle T_{W}\phi^{*}T_{V}^{*}f,
f\rangle\\
&=&\langle f, T_{V}\phi T_{W}^{*}f\rangle=\langle f,f \rangle,\ \
f\in\mathcal{H},
\end{eqnarray*}
i.e. fusion frame $\{W_i\}_{i\in I}$ is also a dual for
$\{V_i\}_{i\in I}$.
\\For the proof of moreover part, since $\{W_{i}\}_{i\in I}$ and
$\{V_{i}\}_{i\in I}$ are fusion Riesz bases, by Theorem \ref{rie},
$T_{W}$ and $T_{V}^{*}$ are invertible. So we deduce the proof by
(\ref{u.phi.T*}) and (\ref{khan}).
\end{proof}
\begin{cor}
Let $\{f_{i}\}_{i\in I}\subseteq \mathcal{H}$ and
$W_{i}=\overline{span}_{i\in I}\{f_{i}\}$ for each $i\in I$.
Suppose that $\{(W_{i}, 1)\}_{i\in I}$ is a tight fusion frame for
$\mathcal{H}$. Then $\{(W_{i}, 1)\}_{i\in I}$ is also a dual
fusion frame of $\{(S_{W}^{-1}W_{i}, 1)\}_{i\in I}.$
\end{cor}
One of the important results in the duality of discrete frames is
that every Riesz basis has just a unique dual (canonical dual) and
that dual is Riesz basis as well. But the following example shows
that this property is not confirmed in fusion Riesz bases.
\begin{ex}\label{dual riesz}
Consider
\begin{equation*}
W_{1}=span\{(1, 0, 0)\},\ \ W_{2}=span\{(1, 1, 0)\}, \ \
W_{3}=span\{(0, 0, 1)\},
\end{equation*}
Then $\{(W_{i}, 1)\}_{i=1}^{3}$ is a fusion frame for
$\mathbb{R}^{3}$ with bounds $1-\frac{\sqrt{2}}{2}$ and $2$, and
the frame operator
\[S_W=\begin{bmatrix}
3/2&1/2&0\\
1/2&1/2&0\\
0&0&1\\
\end{bmatrix}.\]
It is not difficult to see that $\{(W_{i}, 1)\}_{i=1}^{3}$ is a
fusion Riesz basis and its canonical dual can be given with
\begin{equation*}
S^{-1}_WW_{1}=span\{(1, -1, 0)\},\ \ S^{-1}_WW_{2}=span\{(0, 1,
0)\}, \ \ S^{-1}_WW_{3}=span\{(0, 0, 1)\}.
\end{equation*}
To construct an alternate dual consider
\begin{equation*}
V_{1}=\mathbb{R}^2\times\{0\},\ \ V_{2}=S^{-1}_WW_{2}\ \
V_{3}=S^{-1}_WW_{3}.
\end{equation*}
Then $\{(V_{i}, 1)\}_{i=1}^{3}$ is a fusion frame for
$\mathbb{R}^{3}$. Moreover, if $f=(a,b,c)$ then
\begin{eqnarray*}
\sum_{i=1}^3 \pi_{V_i}S_W^{-1}\pi_{W_i}f&=&\pi_{V_1}(a,-a,0)+\pi_{V_2}(0,a+b,0)+\pi_{V_3}(0,0,c)\\
&=&(a,b,c)=f.
\end{eqnarray*}
Hence, the fusion Riesz basis $\{(W_{i}, 1)\}_{i=1}^{3}$ has more
than one dual and the second dual is not a fusion Riesz basis.
\end{ex}
\section{More results on dual construction}
Let $\{(W_{i},\omega_{i})\}_{i\in I}$ be a fusion frame for
$\mathcal{H}$. By considering a frame for each subspace $W_i$ we
can construct a discrete frame for $\mathcal{H}$. We begin with
the following key theorem.
\begin{thm}\label{relation fus.dis}\cite{Cas04}
For each $i\in I$ let $\omega_i>0$ and let $\{f_{i, j}\}_{j\in
J_{i}}$ be a frame sequence in $\mathcal{H}$ with the frame bounds
$A_i$ and $B_i$. Define $W_i=\overline{span}_{j\in
J_{i}}\{f_{i,j}\}$ for all $i\in I$ and assume that
\begin{eqnarray*}
0<A=inf_{i\in I}A_i\leq B=sup_{i\in I}B_i<\infty.
\end{eqnarray*}
Then $\{\omega_{i}f_{i, j}\}_{i\in I, j\in J_{i}}$ is a frame for
$\mathcal{H}$ if and only if $\{(W_{i},\omega_{i})\}_{i\in I}$ is
a fusion frame for $\mathcal{H}$.
\end{thm}
In this paper, we call $F_i=\{f_{i, j}\}_{j\in J_{i}}$, $i\in I$,
local frames of $W_{i}$ and $\{\omega_{i}f_{i, j}\}_{i\in I, j\in
J_{i}}$, the associated discrete frames of $\mathcal{H}$, which
satisfy in above theorem.

Our aim in this section is to study the relation between the duals
of fusion frames, local frames and the associated discrete frames
of $\mathcal{H}$. In particular, in the following theorem we
investigate the relation between the duals of local frames of
$W_{i}$ with the associated discrete frames of $\mathcal{H}$.
\begin{thm}
Let $\{(W_{i},\omega_{i})\}_{i\in I}$ be a fusion frame for
$\mathcal{H}$ with local frames $\{f_{i, j}\}_{j\in J_{i}}$ for
each $i\in I$. If $\{g_{i, j}\}_{j\in J_{i}}$ is a dual frame of
$\{f_{i, j}\}_{j\in J_{i}}$, then $\{w_{i}f_{i, j}\}_{i\in I, j\in
J_{i}}$ is a frame for $\mathcal{H}$ with dual frame
$\{w_{i}S_{W}^{-1}(g_{i, j})\}_{i\in I, j\in J_{i}}$.
\end{thm}
\begin{proof}
Since $\{g_{i, j}\}_{j\in J_{i}}$ is dual frame of $\{f_{i,
j}\}_{j\in J_{i}}$ for $W_{i}$ for each $i\in I$, we obtain
\begin{eqnarray*}
\pi_{W_{i}}(f)=\sum_{j\in J_{i}}\langle\pi_{W_{i}}(f), f_{i,
j}\rangle g_{i, j}=\sum_{j\in J_{i}}\langle f, f_{i, j}\rangle
g_{i, j},\ \ f\in \mathcal{H},\ i\in I.
\end{eqnarray*}
Therefore,
\begin{eqnarray*}
\sum_{i\in I}\sum_{j\in J_{i}}\langle f, w_{i}f_{i, j}\rangle
S_{W}^{-1}(w_{i}g_{i, j})&=&S_{W}^{-1}\sum_{i\in
I}w_{i}^{2}\pi_{W_{i}}f\\&=&S_{W}^{-1}S_{W}f=f.
\end{eqnarray*}
\end{proof}
Suppose that $\{(W_{i},\omega_{i})\}_{i\in I}$ is a fusion frame
for $\mathcal{H}$ and $S_{F_{i}}$ is the frame operator of local
frames $F_{i}$ for each $i\in I$. Now the question is whether the
canonical dual of each frame $F_{i}$ is also a frame for the
canonical dual of $\{(W_{i},\omega_{i})\}_{i\in I}$. The following
example shows that the answer is not true in general.
{{\begin{ex}\label{exam in R3} Let $I=\{1, 2, 3, 4\}.$
Consider
\begin{equation*}
W_{1}=span\{(1, 0, 0)\},\ \ W_{2}=span\{(1, 1, 0)\},
\end{equation*}
\begin{equation*}
 W_{3}=span\{(0,
1, 0)\},\ \ W_{4}=span\{(0, 0, 1)\},
\end{equation*}
and $w_{1}=w_{3}=w_{4}=1, w_{2}=\sqrt{2}.$ Then by Example 3.1 in
\cite{Ami12}, $\{(W_{i},\omega_{i})\}_{i\in I}$ is a fusion frame
for $\mathbb{R}^{3}$. Let $f_{1, 1}=(1, 0, 0), f_{2, 1}=(1, 1, 0),
f_{3, 1}=(0, 1, 0)$ and $f_{4, 1}=(0, 0, 1)$. It is clear that
$\{f_{i, 1}\}$ is a frame for $W_{i}$ for each $i\in I$. Suppose
that $S_{W}$ is the frame operator of  $\{(W_{i},
\omega_{i})\}_{i\in I}$ and $S_{F_{i}}$ is the frame operator of
$\{f_{i, 1}\}$ for each $i\in I$. A straightforward calculation
shows that
\[S_{W}=\begin{bmatrix}
2&1&0\\
1&2&0\\
0&0&1\\
\end{bmatrix}\]
and the subspaces
\begin{equation*}
S_{W}^{-1}W_{1}=span\{(\frac{2}{3}, \frac{-1}{3}, 0)\},\ \
S_{W}^{-1}W_{2}=span\{(\frac{1}{3}, \frac{1}{3}, 0)\},
\end{equation*}
\begin{equation*}
S_{W}^{-1}W_{3}=span\{(\frac{-1}{3}, \frac{2}{3}, 0)\},\ \
S_{W}^{-1}W_{4}=span\{(0, 0, 1)\},
\end{equation*}
with the weights $\{\omega_i\}_{i\in I}$ is the canonical dual of
$\{(W_{i},w_{i})\}_{i\in I}$.  Moreover, if we take
\begin{equation*}
g_{1, 1}=(1, 0, 0),\ g_{2, 1}=(\frac{1}{2}, \frac{1}{2}, 0),\
g_{3, 1}=(0, 1, 0),\ g_{4, 1}=(0, 0, 1),
\end{equation*}
then $\{g_{i, 1}\}$ is the canonical dual of $\{f_{i, 1}\}$ for
each $i\in I$. However, $\{g_{i, 1}\}$ is not a frame for
$S_{W}^{-1}W_{i}$ for each $i\in I$.
\end{ex}}}
The following example shows that there is no significant relation
between the duals of fusion frames and their associated discrete
frames, i.e. if $\{(V_{i}, \nu_{i})\}_{i\in I}$ is a dual of
$\{(W_{i}, \omega_{i})\}_{i\in I}$, then it is not necessary that
their associated discrete frames be dual of each other.
 {{\begin{ex} Let
\begin{equation*}
W_{1}=span\{(1, 0, 0)\},\ \ W_{2}=span\{(0, 1, 0)\},\ \
W_{3}=span\{(0, 0, 1)\},
\end{equation*}
\begin{equation*}
W_{4}=span\{(0, 1, 0)\},\ \ W_{5}=span\{(1, 0, 0)\},\ \
W_{6}=span\{(0, 0, 1)\},
\end{equation*}
and
\begin{equation*}
V_{1}=span\{(1, 0, 0)\},\ \ V_{2}=span\{(0, 1, 0)\},\ \
V_{3}=span\{(0, 0, 1)\},
\end{equation*}
\begin{equation*}
V_{4}=span\{(0, 0, 1)\},\ \ V_{5}=span\{(0, 1, 0)\},\ \
V_{6}=span\{(1, 0, 0)\}.
\end{equation*}
Then $\{(W_{i}, 1)\}_{i=1}^{6}$ is a fusion frame for
$\mathbb{R}^{3}$ with an alternate dual $\{(V_{i},
2)\}_{i=1}^{6}$.
\\Consider
\begin{equation*}
f_{1, 1}=f_{5, 1}=(1, 0, 0), \ \ f_{2, 1}=f_{4, 1}=(0, 1, 0),\ \
f_{3, 1}=f_{6, 1}=(0, 0, 1),
\end{equation*}
and
\begin{equation*}
g_{1, 1}=g_{6, 1}=(2, 0, 0),\ \  g_{2, 1}=g_{5, 1}=(0, 2, 0),\ \
g_{3, 1}=g_{4, 1}=(0, 0, 2).
\end{equation*}
Then $\{f_{i, 1}\}_{i=1}^{6}$ and $\{g_{i, 1}\}_{i=1}^{6}$ are
frames for $\mathbb{R}^{3}$, but they are not dual of each other.
\end{ex}}}
In the following proposition we give a necessary condition to
elucidate duals of fusion frames.
\begin{prop}
Let $\{(W_{i},1)\}_{i\in I}$ be a fusion frame for $\mathcal{H}$
and $\{(V_{i},1)\}_{i\in I}$ be a Parseval fusion frame for
$\mathcal{H}$. Suppose that $W_{k}\perp S_{W}^{-1}V_{i}$ for each
$i\neq k$. Then $\{(V_{i},1)\}_{i\in I}$ is an alternative dual of
$\{(W_{i},1)\}_{i\in I}$.
\end{prop}
\begin{proof}
Since $\{(V_{i},1)\}_{i\in I}$ is a Parseval fusion frame, we have
\begin{eqnarray}\label{2}
\nonumber f&=&S_{V}f\\\nonumber&=&\sum_{i\in
I}\pi_{V_{i}}(S_{W}^{-1}S_{W}f)\\&=&\sum_{i\in
I}\pi_{V_{i}}S_{W}^{-1}\sum_{k\in
I}\pi_{W_{k}}f\\\nonumber&=&\sum_{i\in
I}\pi_{V_{i}}S_{W}^{-1}\pi_{W_{i}}f+\sum_{i\in
I}\pi_{V_{i}}S_{W}^{-1}(\sum_{k\in I, k\neq i}\pi_{W_{k}}f),\ \
f\in\mathcal{H}.
\end{eqnarray}
By Lemma \ref{har}, we have
\begin{equation}\label{si}
\sum_{i\in I}\sum_{k\in I, k\neq
i}\pi_{V_{i}}S_{W}^{-1}\pi_{W_{k}}f=\sum_{i\in I}\sum_{k\in I,
k\neq i}\pi_{V_{i}}S_{W}^{-1}\pi_{S_{W}^{-1}V_{i}}\pi_{W_{k}}f=0.
\end{equation}
So we get the proof by (\ref{2}) and (\ref{si}).
\end{proof}
\begin{prop}
Let $\{(W_{i},\omega_i)\}_{i\in I}$ and $\{(V_{i},\nu_i)\}_{i\in
I}$ be fusion frames for $\mathcal{H}$. Suppose that $W_{i}\perp
V_{i}$ for each $i\in I$. Then $\{(S_{W}V_{i},\nu_i)\}_{i\in I}$
can not be an alternative dual of $\{(W_{i},\omega_i)\}_{i\in I}$.
\end{prop}
\begin{proof}
By Proposition 1.1 in \cite{Cas04}, $\{( S_{W}V_{i},\nu_i)\}_{i\in
I}$ is a fusion frame for $\mathcal{H}$. By Lemma \ref{har} we
have
\begin{eqnarray*}
\sum_{i\in
I}\omega_{i}\nu_{i}\pi_{S_{W}V_{i}}S_{W}^{-1}\pi_{W_{i}}f&=&\sum_{i\in
I}\omega_{i}\nu_{i}\pi_{S_{W}V_{i}}S_{W}^{-1}\pi_{V_{i}}\pi_{W_{i}}f=0,
\ \ f\in\mathcal{H}.
\end{eqnarray*}
\end{proof}
In the rest of the paper we try to characterize the duals of
fusion frames. We first discuss the Riesz case. Let
$\{(W_i,\omega_i)\}_{i\in I}$ be a Riesz decomposition of
$\mathcal{H}$ and $\{(V_i,\nu_i)\}_{i\in I}$ be its dual.
Associated to the canonical dual
$\{(S_{W}^{-1}W_i,\omega_i)\}_{i\in I}$ we can consider the
operator $\phi_1:(\sum_{i\in I}\oplus W_i)_{\ell^{2}}\rightarrow
(\sum_{i\in I}\oplus S_{W}^{-1}W_i)_{\ell^{2}}$ given by
\begin{eqnarray*}
\quad\phi_1(\{f_i\}_{i\in
I})=\{\pi_{S_{W}^{-1}W_i}S_{W}^{-1}f_i\}_{i\in I}.
\end{eqnarray*}
Applying (\ref{u.phi.T*}) and Theorem \ref{rie} we conclude that
$T_{V}\phi=T_{S_{W}^{-1}W}\phi_1$, where $T_{S_{W}^{-1}W}$ is the
synthesis operator of $\{(S_{W}^{-1}W_{i}, \omega_{i})\}_{i\in
I}$. It follows easily that
\begin{eqnarray*}
\pi_{V_i}S_{W}^{-1}f_i=S_{W}^{-1}f_i, \quad i\in I,f_i\in W_i,
\end{eqnarray*}
or equivalently
\begin{eqnarray*}
S_{W}^{-1}W_i\subseteq V_i, \quad i\in I.\end{eqnarray*}The
following example shows that unfortunately, we can not
characterize the duals of fusion frames by the duals of their
associated discrete frames and the first part of Proposition
\ref{ar}.
{{\begin{ex} Consider the fusion frame
$\{(W_i,\omega_i)\}_{i=1}^4$ introduced in Example \ref{exam in
R3}. By Theorem \ref{relation fus.dis} the sequence
\begin{equation*}
\{\omega_if_{i,1}\}_{i=1}^4=\{(1, 0, 0),\sqrt{2}(1, 1, 0),(0, 1,
0),(0, 0, 1)\}\end{equation*} is frame for $\mathbb{R}^3$ with the
frame operator
\[S_F=\begin{bmatrix}
3&2&0\\
2&3&0\\
0&0&1\\
\end{bmatrix}.\]
Denote its canonical dual by $\{g_{i,1}\}_{i=1}^4$. Then
\begin{equation*}
\{g_{i,1}\}_{i=1}^4=\{\frac{1}{5}(3,-2, 0),\frac{\sqrt{2}}{5}(1,
1, 0),\frac{1}{5}(-2,3, 0),(0, 0, 1)\}.\end{equation*} Consider
\begin{equation*}
V_{1}=span\{(3, -2, 0)\}\ \ V_{2}=span\{(1, 1, 0)\},
\end{equation*}
\begin{equation*}
V_{3}=span\{(-2, 3, 0)\},\ \ V_{4}=span\{(0, 0, 1)\},
\end{equation*}
and $\nu_{1}=\nu_{3}=\frac{1}{5},\ \ \nu_{2}=\frac{\sqrt{2}}{5},\
\ \nu_{4}=1$, then $\{(V_{i}, \nu_{i})\}_{i=1}^{4}$ is a fusion
frame for $\mathbb{R}^{3}$. But $\{(V_{i}, \nu_{i})\}_{i=1}^{4}$
is not an alternative dual of $\{(W_{i}, \omega_{i})\}_{i=1}^{4}$
and vise-versa.
\end{ex}}
We give an explicit construction of a dual fusion frame in the
following theorem.
\begin{thm} Let $\{(W_i, 1)\}_{i\in I}$ be a fusion frame for $\mathcal{H}$ and $\{h_i\}_{i\in I}$ be a Bessel sequence of normalized vectors such that
$h_i\in (S_W^{-1}W_i)^{\perp}$. Take
\begin{eqnarray*}
V_i=S_W^{-1}W_i + Z_i,\quad i\in I,
\end{eqnarray*}
where $Z_i$ is the 1-dimensional subspace generated by $h_i$. Then
$\{(V_i, 1)\}_{i\in I}$ is a dual for $\{(W_i, 1)\}_{i\in I}$.
\end{thm}
\begin{proof} First, it is not difficult to see that
\begin{eqnarray*}
\pi_{Z_i}f=\langle f,h_i\rangle h_i,\quad i\in I, f\in
\mathcal{H}.
\end{eqnarray*}
Now by using 8.12 of \cite{Gre81} and Corollary 2.5 of
\cite{Gav07} we conclude that
\begin{eqnarray*}
\sum_{i\in I}\|\pi_{V_i}f\|^2&=&\sum_{i\in I}\|\pi_{S_W^{-1}W_i}f+\pi_{Z_i}f\|^2\\
&\leq &\sum_{i\in I}\|\pi_{S_W^{-1}W_i}f\|^2 + \sum_{i\in I}|\langle f,h_i\rangle|^2\\&+& 2\left(\sum_{i\in I}\|\pi_{S_W^{-1}W_i}f\|^2\right)^{\frac{1}{2}}\left(\sum_{i\in I}|\langle f,h_i\rangle|^2\right)^{\frac{1}{2}}\\
&\leq&\left(B\|S_W\|^2\|S_W^{-1}\|^2+D+2\sqrt{BD}\|S_W\|\|S_W^{-1}\|\right)\|f\|^2,
\end{eqnarray*}
where $B$ and $D$ are the frame bounds of $\{(W_i, 1)\}_{i\in I}$
and $\{h_i\}_{i\in I}$, respectively. Hence, $\{(V_i, 1)\}_{i\in
I}$ is a Bessel fusion frame. Moreover, by Lemma \ref{har} we have
\begin{eqnarray*}
\sum_{i\in I}\pi_{V_i}S_{W}^{-1}\pi_{W_i}f&=&\sum_{i\in I}\pi_{S_W^{-1}W_i}S_{W}^{-1}\pi_{W_i}f+\sum_{i\in I}\pi_{Z_i}S_{W}^{-1}\pi_{W_i}f\\
&=& f,\ \ \ \ f\in\mathcal{H}.
\end{eqnarray*}
\end{proof}
\begin{rem}  The above theorem gives us a very simple method to construct duals of finite fusion frames.
More precisely, let $\{(W_i, 1)\}_{i\in I}$ be a finite fusion
frame for $\mathcal{H}$. Take
\begin{eqnarray*}
V_i=\left\{\begin{array}{cc}
                        S_W^{-1}W_i &\mbox{for $(S_W^{-1}W_i)^{\perp}= \emptyset$ }\\
                       S_W^{-1}W_i\oplus Z_i&\mbox{for $(S_W^{-1}W_i)^{\perp}\neq \emptyset$},\\
                     \end{array}\right. \end{eqnarray*}
where $Z_i$ is a 1-dimensional subspace of
$(S_W^{-1}W_i)^{\perp}$. To illustrate this algorithm, let us
consider the fusion frame $\{W_i\}_{i\in I}$ in Example \ref{dual
riesz}. Clearly
\begin{equation*}
\left(S^{-1}_WW_{1}\right)^{\perp}=span\{(1, 1, 0),(0, 0, 1)\},
\end{equation*}
\begin{equation*}
 \left(S^{-1}_WW_{2}\right)^{\perp}=span\{(1, 0,
0), (0, 0, 1)\},
\end{equation*}
\begin{equation*}
 \left(S^{-1}_WW_{3}\right)^{\perp}=span\{(1, 0, 0), (0, 1, 0)\}.
\end{equation*}
Therefore, we can introduce some duals:
\begin{equation*}
(i)\quad V_{1}=span\{(1, -1, 0), (0 , 0, 1)\},\ \
V_{2}=\mathbb{R}^2\times \{0\},\ \ V_{3}=\{0\}\times\mathbb{R}^2.
\end{equation*}
\begin{equation*}
(ii)\quad V_{1}=span\{(1, -1, 0), (0 , 0, 1)\},\ \
V_{2}=\{0\}\times\mathbb{R}^2,\ \ V_{3}=\{0\}\times\mathbb{R}^2.
\end{equation*}
\begin{equation*}
(iii)\quad V_{1}=span\{(1, -1, 0), (1 , 1, 0)\},\ \
V_{2}=\mathbb{R}^2\times\{0\} ,\ \ V_{3}=\{0\}\times\mathbb{R}^2.
\end{equation*}
\begin{equation*}
(iv)\quad V_{1}=span\{(1, -1, 0), (1 , 1, 0)\},\ \
V_{2}=\{0\}\times\mathbb{R}^2 ,\ \ V_{3}=\{0\}\times\mathbb{R}^2.
\end{equation*}
\end{rem}}

\bibliographystyle{amsplain}

\end{document}